\documentclass[11pt]{article}
\usepackage{amsmath,amssymb,amsthm}
\usepackage{mathrsfs}
\usepackage{natbib}

\newcommand{\qbin}[2]{\genfrac{[}{]}{0pt}{}{#1}{#2}}
\newcommand{\gp}[3]{\qbin{#1}{#2}_{#3}}
\newcommand{\Bin}{\mathrm{B\textsc{in}}}

\newcommand{\inv}{\mathrm{inv}}

\newcommand{\Mult}{\mathrm{M\textsc{ult}}}
\newcommand{\Comp}{\mathscr{C}}

\newcommand{\Yv}{\mathbf{Y}}
\newcommand{\yv}{\mathbf{y}}
\newcommand{\Xv}{\mathbf{X}}
\newcommand{\xv}{\mathbf{x}}

\newtheorem{theorem}{Theorem}

\newtheorem{lemma}[theorem]{Lemma}

\theoremstyle{remark}
\newtheorem{remark}[theorem]{Remark}

\numberwithin{theorem}{section}
\numberwithin{equation}{section}

\title{A Refinement of the 
Multinomial Distribution\\with Application}
\author{Andrew V. Sills\\
Department of Mathematical Sciences,\\
Georgia Southern University, Statesboro and Savannah, Georgia, USA\\
email: asills@georgiasouthern.edu}
\date{\today}

\begin{document}
\maketitle

\centerline{Key words: multinomial distribution; $q$-multinomial theorem,
permutation inversion statistic}
\centerline{MSC: 60E05, 05A05}

\begin{abstract}
A refinement of the multinomial distribution is presented where the number of
inversions in the sequence of outcomes is tallied.  This refinement of the multinomial
distribution is its joint distribution with the number of inversions in the 
accompanying experiment.  An application of this additional
information is described in which the number of inversions acts as a proxy 
measure of homogeneity (or lack thereof) in the sequence of outcomes.
\end{abstract}


\section*{Competing Interests Statement}
The author has no competing interests to declare that are relevant to the content of this article.
No funds, grants, or other support was received.

\section{Preliminaries} 
The author sets out to establish a refinement of the classical multinomial 
probability distribution that generalizes his earlier work~\citep{S21} and here
 draws inspiration from the non-commutative $q$-multinomial
theorem of~\cite{F07}.  This multinomial distribution generalization is
different that of~\cite{C21}, which makes use of the commutative
$q$-multinomial theorem.

Recall the  classical multinomial experiment:
\begin{itemize}
   \item The experiment consists of $n$ independent, identical trials;
   \item there are $k$ possible outcomes on each trial, labeled outcome $1,2,\dots, k$;
   \item on a given trial, the probability that outcome $i$ occurs is the constant $p_i$.
\end{itemize}
Define the random variable $Y_i$ as the number of trials that result in outcome $i$.
Let $\Yv = (Y_1, Y_2, \dots, Y_k)^T$ and $\yv = (y_1, y_2, \dots, y_k)^T$.

The joint pmf for $\Yv$ is
\[ P(\Yv=\yv) = \binom{y_1 + y_2 + \cdots + y_k}{y_1, y_2, \dots, y_k} 
p_1^{y_1} p_2^{y_2} \cdots p_k^{y_k}, \]
for all nonnegative integers $y_1, y_2, \dots, y_k$ that sum to $n$, and $0$
otherwise;
where 
\begin{equation}\label{mc}
  \binom{y_1 + y_2 + \cdots + y_k}{y_1, y_2, \dots, y_k}
  = \frac{(y_1 + y_2 + \cdots + y_k)!}{y_1! y_2! \cdots y_k!} 
 \end{equation} is
the multinomial co\"efficient and
$ \sum_{i=1}^k p_i = 1. $ 
In this case, we may write $\Yv\sim\Mult(n; p_1, p_2, \dots, p_{k-1})$ noting
that once $p_1, p_2, \dots, p_{k-1}$ are specified, it must be the case that
$p_k = 1 - \sum_{i=1}^{k-1} p_i$.

A \emph{weak $k$-composition of $n$} is a $k$-tuple of nonnegative integers
whose components sum to $n$.  For example, 
there are $10$ weak $3$-compositions of $3$:
namely
\[ (0,0,3) \quad (0,1,2) \quad (0,2,1) \quad (0,3,0) \quad (1,0,2) \]
\[ (1,1,1) \quad (1,2,0) \quad (2,0,1) \quad (2,1,0) \quad (3,0,0). \]  In general,
the number of weak $k$-compositions of $n$ is $\binom{n+k-1}{k-1}$.
Denote the set of weak $k$-compositions of $n$ by $\Comp_k(n)$.

The weak $k$-compositions of $n$ comprise the support of the random
vector $\Yv$, where $\Yv\sim\Mult(n; p_1, p_2, \dots, p_{k-1})$.

A standard reference on integer compositions is~\cite{HM09}.

The \emph{$q$-multinomial co\"efficient} (or 
\emph{Gaussian multinomial co\"efficient}) is given by
\begin{equation} \label{qmc}
   \gp{y_1 + y_2 + \cdots + y_k}{y_1, y_2, \dots, y_k}{q}  = 
   \frac{(q)_{y_1 + y_2+ \cdots + y_k}}{(q)_{y_1} (q)_{y_2} \cdots (q)_{y_k}},
\end{equation}
where $(q)_j := (1-q)(1-q^2)\cdots (1-q^j).$
 Note that the $q$-multinomial co\"efficient in~\eqref{qmc} is a polynomial in $q$
of degree \[ \sum_{1\leq i< j \leq k} y_i y_j  \] and if $q$ is set to $1$ in~\eqref{qmc}, then the ordinary multinomial
co\"efficient~\eqref{mc} is recovered.

Let us denote the multiset
$M = (a_1, a_2, \dots, a_n)$, where the first $m_1$ elements equal $1$,
the next $m_2$ elements equal $2$, the next $m_3$ elements equal $3$,
\dots, the last $m_k$ elements equal $k$, and $n = m_1 + m_2 + \cdots + m_k$,
by the notation $\{ 1^{m_1} 2^{m_2} \cdots k^{m_k} \}$.
A permutation $(\xi_i, \xi_2, \dots, \xi_n)$ of $M$ is said to have $i$ 
\emph{inversions} if there are exactly $i$ pairs $(\xi_i, \xi_j)$ such that $i<j$ 
and $\xi_i > \xi_j$.    The number of permutations of $M$ having exactly $i$
inversions is denoted $\inv(m_1, m_2, \cdots, m_k ; i)$.

 The $q$-multinomial co\"efficient figures prominently in the theory of
permutations of the multiset $\{ 1^{m_1} 2^{m_2} \cdots k^{m_k} \}$.  In fact,
\cite{M17} proved that 
\begin{equation} \label{invgf} 
\sum_{i\geq 0} \inv(y_1, y_2, \dots, y_k ; i) q^i = \
\gp{y_1 + y_2 + \cdots + y_k}{y_1, y_2, \dots, y_k}{q},
\end{equation} i.e. that the $q$-multinomial co\"efficient is the ordinary 
power series generating function for the function that counts the number of
inversions in the multiset $M$ defined above.

  By setting $q=1$ in~\eqref{invgf}, it is immediate that
\begin{equation} \label{numinv}
  \sum_{i\geq 0} \inv(y_1, y_2, \dots, y_k; i) = 
  \binom{y_1 + y_2 + \cdots  y_k}{y_1, y_2, \dots, y_k}.
\end{equation}

  For more on the $q$-multinomial co\"efficients, see section 3.3 of~\cite{A76}.

\section{A refined multinomial distribution with mathematical motivation}
\subsection{A special case}
To motivate the refinement about to be proposed, let us consider in detail the
specific example where $(Y_1, Y_2, Y_3)\sim\Mult(3; p_1, p_2)$; of course,
$Y_3 = 3 - Y_1 - Y_2$ and $p_3 = 1 - p_1 - p_2$.  That is, we are
considering $n=3$ trials with $k=3$ possible outcomes on each
trial.  In the associated 
multinomial experiment,  there are $k^n = 3^3 = 27$ possible outcomes of the 
$n=3$ trials.   Each outcome is enumerated explicitly in the following table:

\begin{center}
\begin{tabular}{ | c | c | c | c | }
\hline
$\yv $      & $P(\Yv = \yv)$ & outcome(s) & $\inv$ \\
\hline \hline
$(0,0,3)$ &  $p_3^3$            &  $333$  & 0 \\
\hline
                &                           & $233$  & 0 \\
$(0,1,2)$  &  $3 p_2 p_3^2$ & $323$   & 1 \\
                &                           & $332$   & 2 \\ 
\hline
                &                           & $223$  & 0 \\
$(0,2,1)$  &  $3 p_2^2 p_3$ & $232$  & 1\\
                &                           & $322$  & 2\\ 
\hline
$(0,3,0)$  & $p_2^3$            & $222$  & 0 \\
\hline
                &                           & $133$  & 0 \\
$(1,0,2)$  &  $p_1 p_3^2$    & $313$  & 1\\
                &                            & $331$ & 2\\
\hline                             
                &                            & $123$  & 0 \\
                &                            & $132$  & 1\\
$(1,1,1)$ &  $6 p_1 p_2 p_3$ & $213$    & 1\\     
                &                            & $231$   & 2 \\
                &                            & $312$   & 2\\
                &                            & $321$    & 3 \\
\hline
                &                           & $122$  & 0 \\
$(1,2,0)$  &  $3 p_1 p_2^2$ & $212$  & 1\\
                &                           & $221$  & 2\\
\hline   
                &                           & $113$  & 0 \\
$(2,0,1)$  &  $3 p_1^2 p_3$ & $131$   & 1 \\
                &                           & $311$   & 2 \\ 
\hline
               &                           & $112$  & 0 \\
$(2,1,0)$  &  $3 p_1^2 p_2$ & $121$   & 1 \\
                &                           & $211$   & 2 \\ 
\hline
$(3,0,0)$  & $p_1^3$               & $111$   & 0 \\
   \hline                                                                       
\end{tabular}
\\
Table 1. The distribution $\Mult(3; p_1, p_2)$ refined with inversion count
\end{center}
  If we ignore the last column, we have the classical multinomial distribution, where
the number of each of the $k$ possible outcomes is recorded over the $n$ trials.
However, the last column suggests that we could additionally keep track of the
number of inversions, where we view the sequence of outcomes across the 
$n = 3$ trials as some permutation of $\{ 1^{y_1} 2^{y_2} 3^{y_3} \}$.  

\subsection{The random variable that counts the number of inversions in the
outcomes of the sequence of trials, and the resulting joint pmf}

Let $I$ denote the random variable that records the number of inversions in
outcome of the multinomial experiment and let $\Xv = (Y_1, Y_2,\dots, Y_k, I)$,
a random vector with $k+1$ components.   Then it seems reasonable to
define the joint pmf 
\begin{equation} \label{YIpmf}
  P(\Xv = \xv) =
  P(\Yv = \yv, I=i)
  = \inv(y_1, y_2, \dots, y_k; i) p_1^{y_1} p_2^{y_2} \cdots p_k^{y_k},
\end{equation} for $\yv \in \Comp_k(n)$ and nonnegative integers $i$; and $0$ otherwise.

\begin{remark}
Actually, the support of $I$ is \emph{not} all nonnegative integers $i$, 
although if a value of $i$
that is out of bounds is selected, the value of $\inv(y_1, \dots, y_k; i)$ will equal $0$,
so no real harm done by extending $i$ over all nonnegative integers.  
For a fixed $n$ and $k$, it appears that the maximum $i$ for
which $\inv(y_1, \dots, y_k; i) > 0$ is $\lfloor (k-1)n^2/(2k) \rfloor$. 
\end{remark}

  Notice that~\eqref{YIpmf} is a mathematically legitimate joint pmf, as
$P(\Yv = \yv, I=i) \geq 0$ for all $\yv$ and $i$, and
\begin{align*}
\sum_{\yv\in\Comp_k(n)} \sum_{i\geq 0} P(\Yv = \yv, I=i) 
& = \sum_{\yv\in\Comp_k(n)} \sum_{i\geq 0}
 \inv(y_1,  \dots, y_k; i) p_1^{y_1} \cdots p_k^{y_k} \\
 & = \sum_{\yv\in\Comp_k(n)} p_1^{y_1} \cdots p_k^{y_k} \sum_{i\geq 0}
 \inv(y_1,  \dots, y_k; i) \\
 & = \sum_{\yv\in\Comp_k(n)} p_1^{y_1} \cdots p_k^{y_k} 
 \binom{y_1 + \cdots + y_k}{y_1, \dots, y_k} \\ & = 1.
\end{align*}

From this perspective, the previously considered $n=k=3$ case can be tabulated
as follows, where each cell contains the relevant sequence of outcomes, and
the probability of that particular value of $(Y_1, Y_2, Y_3)$ and $I$:

\begin{center}
\begin{tabular}{c|cccc}
\hline \hline
    &     &   $I$ &  & \\
   $\yv$ & 0 & 1 & 2 & 3  \\
\hline\hline
(0,0,3) &  333         & --- & --- & --- \\
   &  $p_3^3$ & 0  & 0   & 0    \\
\hline
(0,1,2)  &  233        & 323  & 332 & ---  \\
    & $p_2 p_3^2$ & $p_2 p_3^2$ & $p_2 p_3^2$ &  $0$ \\
\hline    
(0,2,1)  &  223        & 232  & 322 & ---  \\
    & $p_2^2 p_3$ & $p_2^2 p_3$ & $p_2^2 p_3$ &  $0$ \\
\hline
(0,3,0)  & 222        & --- & --- & --- \\
            & $p_2^3$ &  0  &  0 &  0\\
 \hline    
(1,0,2) & 133 & 313 & 331 & --- \\
            & $p_1 p_3^2$       &  $p_1 p_3^2$ &  $p_1 p_3^2$ &  0\\
 \hline  
 (1,1,1) & 123 & 132, 213 & 231, 312 & 321 \\
             & $p_1 p_2 p_3$ &     $2p_1 p_2 p_3$    & $2p_1 p_2 p_3$  & $p_1 p_2 p_3$ \\
 \hline   
(1,2,0) & 122 & 212 & 221 & --- \\
            & $p_1 p_2^2$ &   $p_1 p_2^2$ &  $p_1 p_2^2$ &    0 \\
 \hline
(2,0,1)  & 113 & 131 & 311 & --- \\
             & $p_1^2 p_3$     & $p_1^2 p_3$  & $p_1^2 p_3$ & 0 \\
\hline       
(2,1,0) & 112 & 121 & 211 & --- \\
           & $p_1^2 p_2$  & $p_1^2 p_2$ & $p_1^2 p_2$ & 0 \\
 \hline
 (3,0,0) & 111 & --- & --- & --- \\
             & $p_1^3$ & 0 & 0 & 0 \\
             \hline          
\end{tabular}\\
Table 2. The joint distribution of $(\mathbf{Y}, I)$.
\end{center}
Notice how the co\"efficient of probability expression of the cell in row 
$\Yv = (y_1,y_2,y_3)$
and column $I=i$ equals the co\"efficient of $q^i$ in the 
expansion of the
$q$-multinomial
co\"efficient $\gp{y_1+y_2+y_3}{y_1,y_2,y_3}{q}. $ 
For instance, $\gp{3}{1,1,1}{q} = 1 + 2q + 2q^2 + q^3$, as a
demonstration of Eq.~\eqref{invgf}.   

\subsection{Farsi's noncommutative $q$-multinomial theorem}
For additional mathematical motivation, 
recall the non-commutative $q$-multinomial theorem of~\cite{F07}.  In particular, 
consider~\cite[p. 1539, Example 2.8]{F07}, with each $\rho_j$ set 
equal to $q$, and notation adjusted for our present purposes: Let $n$ denote a 
positive integer and let
$x_1, x_2, \dots, x_k$ denote multiplicatively noncommutative indeterminates
that satisfy the $q$-commutation relation 
\begin{equation} \label{qcommute} 
x_j x_i = q x_i x_j  \mbox{  whenever } i<j; \end{equation} then
\begin{equation}\label{qMT}
(x_1 + x_2 + \cdots + x_k)^n = \sum_{\yv\in\Comp_k(n)}
\gp{y_1 + y_2 + \cdots + y_k}{y_1, y_2, \dots, y_k}{q}
 x_1^{y_1} x_2^{y_2} \cdots x_k^{y_k}.
\end{equation}
The left member of~\eqref{qMT}, upon expansion, encodes all
possible permutations of $\{ 1^{y_1} 2^{y_2} \cdots k^{y_k} \}$ for all ordered
$k$-tuples of nonnegative integers $(y_1, \dots, y_k)$ that sum to $n$, i.e.
all weak $k$-compositions of $n$, as a sum of $k^n$ terms, each 
consisting of a noncommutative product of  $n$ $x_i$'s,
in every possible order.    The right member expresses an algebraically 
equivalent expression, taking into account the $q$-commutation relation
\eqref{qcommute}, where the $x_i$'s appear in their ``natural'' order, i.e.
with increasing subscripts.   
   In the $n=k=3$ case, the left member of~\eqref{qMT} expands as these
   $27$ terms:
\begin{multline}
 x_3 x_3 x_3 + x_2 x_3 x_3 + x_3 x_2 x_3 + x_3 x_3 x_2 + x_2 x_2 x_3 + x_2 x_3 x_2 +
 x_3 x_2 x_2 + x_2 x_2 x_2 \\
 + x_1 x_3 x_3 + x_3 x_1 x_3 + x_3 x_3 x_1 + x_1 x_2 x_3 + x_1 x_3 x_2 + x_2 x_1 x_3
 + x_2 x_3 x_1 + x_3 x_1 x_2 \\
 + x_3 x_2 x_1 + x_1 x_2 x_2 + x_2 x_1 x_2 + x_2 x_2 x_1 + x_1 x_1 x_3 + x_1 x_3 x_1
  + x_3 x_1 x_1 \\
  + x_1 x_1 x_2 + x_1 x_2 x_1 + x_2 x_1 x_1 + x_1 x_1 x_1,
\end{multline}
where the set of possible outcomes is enumerated by following the subscripts.
In contrast, the right member of~\eqref{qMT} in the $n=k=3$ case expands as
\begin{multline}
  (1) x_3^3  + (1+q+q^2) x_2 x_3^2 + (1+ q + q^2) x_2^2 x_3 + (1) x_2^2 + 
  (1+q+q^2) x_1 x_3^2 \\
  + (1+2q + 2q^2 + q^3) x_1 x_2 x_3 + (1+q+q^2) x_1 x_2^2 + (1+q+q^2) x_1^2 x_3\\
  + (1+q+q^2) x_1^2 x_2 + (1)x_1^3.
\end{multline} Here, the number of inversions required to transform
$x_1^{y_1} x_2^{y_2} \dots x_k^{y_k}$ into the permuted form as required by
\eqref{qcommute} is recorded in each exponent of $q$.

  If~\eqref{invgf} is applied to the right member of~\eqref{qMT}, the result is:
 \begin{multline} \label{convert}
 \sum_{\yv\in\Comp_k(n)}
\gp{y_1 + y_2 + \cdots + y_k}{y_1, y_2, \dots, y_k}{q}
 x_1^{y_1} x_2^{y_2} \cdots x_k^{y_k} \\ =
  \sum_{\yv\in\Comp_k(n)}
\sum_{i\geq 0} \inv(y_1, \dots, y_k; i) q^i
 x_1^{y_1} x_2^{y_2} \cdots x_k^{y_k} 
 \end{multline}

Consider the generic summand of the right member of~\eqref{convert}.
Substitute $x_i = p_i$ for $i = 1, 2, \dots, k$, and take $q=1$.  
The resulting expression is the joint pmf of $\Yv$ and $I$ 
summed over its support~\eqref{YIpmf}.

\subsection{Some Marginal distributions}
  Clearly, we have $\Yv\sim\Mult(n; p_1, p_2, \dots, p_{k-1})$ as this was our
starting point before introducing the random variable $I$.  Additionally, 
$Y_j \sim \Bin(n, p_j)$ for each $j$, $1\leq j \leq k$, and
accordingly for any such $i$, the joint pmf of $Y_j$ and $I$ is
\begin{equation}
   P(Y_j = y_j, I=i) =  \inv( y_j, n - y_j   ; i) p_j^{y_j} (1 - p_j)^{n - y_j},
\end{equation} for $0 \leq y_j \leq n$ and $0 \leq i \leq \lfloor n^2/4 \rfloor$, and $0$
otherwise.  
This case was studied in~\cite{S21}.

\section{A summation lemma to aid calculations of moments}

\begin{lemma} \label{lem}
Let $y_1, y_2, \dots, y_k$ be fixed nonnegative integers.
\begin{equation} \label{mom1}
\sum_{i\geq 0} i \cdot \inv(y_1, \dots, y_k; i) 
 = \frac 12\binom{y_1 + \cdots + y_k}{y_1, \cdots, y_k} 
     \sum_{1\leq i < j \leq k} y_i y_j 
\end{equation}
\begin{multline}
\label{mom2}
\sum_{i\geq 0} i^2 \cdot \inv(y_1, \dots, y_k; i) \\
  =  \binom{y_1+\cdots+y_k}{y_1,\cdots,y_k}  \left( 
  \frac{1}{12}  \sum_{1\leq i < j \leq k} y_i y_j 
  + \frac{1}{12} \sum_{1\leq i < j \leq k} y_i^2 y_j
  + \frac{1}{12} \sum_{1\leq i < j \leq k} y_i y_j^2 \right. \\
  + \frac{1}{4} \sum_{1\leq i < j \leq k} y_i^2 y_j^2 
  + \frac{1}{6} \sum_{1\leq h<i<j\leq k} y_h y_i y_j  
  + \frac{1}{2} \sum_{1\leq h<i<j\leq k} y_h^2 y_i y_j \\
  +
      \left. \frac{1}{2} \sum_{1\leq h<i<j\leq k} y_h y_i^2 y_j
    + \frac{1}{2} \sum_{1\leq h<i<j \leq k} y_h y_i y_j^2 
    + \frac{3}{2} \sum_{1\leq g<h<i<j \leq k} y_g y_h y_i y_j \right) \\
 = \frac{1}{12} \binom{y_1+\cdots+y_k}{y_1,\cdots,y_k}  \left\{
   \sum_{1\leq i < j \leq k} y_i y_j 
  +  \left(  \sum_{1 \leq h \leq i < j \leq k}  
                             + \sum_{1 \leq h < i \leq j \leq k} \right) y_h y_i y_j  \right. \\
    + 6  \left( 
        \sum_{1\leq g \leq h < i < j \leq k} 
      +\sum_{1\leq g < h \leq i < j \leq k} 
      + \sum_{1\leq g < h < i \leq j \leq k} \right) y_g y_h y_i y_j  \\
   \left.   +  3\sum_{1\leq g = h < i = j  \leq k} y_g y_h y_i y_j \right\}.                  
\end{multline}
 \end{lemma}

For~\eqref{mom1}, the key is to note that the sum 
$\sum_{i\geq 0} i \cdot \inv(y_1, \dots, y_k; i) $ can be
obtained by differentiating the $q$-multinomial co\"efficient with respect to $q$,
and then setting $q=1$:
\[ \sum_{i\geq 0} i \cdot \inv(y_1, \dots, y_k; i)  = \frac{d}{dq} \gp{y_1+\cdots + y_k}
{y_1, \dots, y_k}{q} \Bigg|_{q=1}. \]
\cite[p. 142]{M15}  computed the $k=2, 3$ cases, and then he
recorded the result for general $k$.

For~\eqref{mom2}, we analogously have 
 \[ \sum_{i\geq 0}  i^2\cdot \inv(y_1, \dots, y_k; i)  
 = \frac{d}{dq} \left( q\frac{d}{dq} \gp{y_1+\cdots + y_k}
{y_1, \dots, y_k}{q}\right) \Bigg|_{q=1}. \]  A full proof
is presented as an appendix to this paper, so as not to disrupt the flow of the current
line of thought.

Notice that for $k<4$, Eq.~\eqref{mom2} simplifies considerably
as some of the multisums
are empty.  Specifically,
\begin{equation}
\label{mom2k2}
\sum_{i\geq 0} i^2 \cdot \inv(y_1, y_2; i) 
  =  \binom{y_1+y_2}{y_1,y_2}  \left( 
  \frac{1}{12} \left( y_1 y_2 + y_1^2 y_2 + y_1 y_2^2
  \right) 
  + \frac{1}{4} y_1^2 y_2^2 \right)  
\end{equation}

\begin{multline}
\label{mom2k3}
\sum_{i\geq 0} i^2 \cdot \inv(y_1, y_2, y_3; i) \\
  =  \binom{y_1+y_2+y_3}{y_1,y_2,y_3}  \left( 
  \frac{1}{12}  \left( \sum_{1\leq i < j \leq 3} y_i y_j 
  + \sum_{1\leq i < j \leq 3} y_i^2 y_j
  + \sum_{1\leq i < j \leq 3} y_i y_j^2 \right) \right. \\ \left.
  + \frac{1}{4} \sum_{1\leq i < j \leq 3} y_i^2 y_j^2 
  + \frac{1}{6} y_1 y_2 y_3 
  + \frac{1}{2} \left( y_1^2 y_2 y_3 
  + y_1 y_2^2 y_3
  + y_1 y_2 y_3^2 \right)\right),  
\end{multline}

\section{The marginal distribution of $I$}
\subsection{The probability mass function and a possible limiting distribution}
The marginal distribution of $I$ is given by
\begin{multline} \label{pmfI}
 P(I = i) = \sum_{(y_1, \dots, y_k)\in\mathscr{C}_k } \inv(y_1,\dots,y_k; i) 
   p_1^{y_1} \cdots p_k^{y_k},  \\ 
   \mbox{ for $i = 0, 1, 2, \dots , \lfloor (k-1)n^2/(2k) \rfloor$},
\end{multline} and $0$ otherwise.
 
   One referee commented that it was not easy to compute directly with Eq.~\eqref{pmfI},
and that it would be nice to have an asymptotic distribution for $I$.
The author worked at this for a while and encountered difficulties, so he decided to
consider an easier but related problem: that of finding an asymptotic distribution for $I$
in the case where all $p_j$ are equal, i.e. $p_1 = p_2 = \cdots = p_k = 1/k$.
In this case~\eqref{pmfI} simplifies (slightly) to
\begin{multline} \label{pmfIeq}
 P(I = i) = \frac{1}{k^n}\sum_{(y_1, \dots, y_k)\in\mathscr{C}_k } \inv(y_1,\dots,y_k; i)  \\ 
   \mbox{ for $i = 0, 1, 2, \dots , \lfloor (k-1)n^2/(2k) \rfloor$},
\end{multline} and $0$ otherwise;
but the problem of counting the inversions, which appear as the co\"efficients of $q^i$ 
in the $q$-multinomial co\"efficient $\gp{n}{y_1, \dots, y_k}{q}$ remains, and is far from trivial.
Since the constant term in $\gp{n}{y_1, \dots, y_k}{q}$ is $1$ for all $n$ and $k$, we may
deduce that
\[ P(I=0) = \frac{1}{k^n} \binom{n+k-1}{n}, \]
but for $i>0$, $P(I=i)$ seems inaccessible for symbolic $n$ and $k$.  Of course, for specific
numeric $n$ and $k$, the computation of the full pmf for $I$ presents no difficulty.

Extensive computations performed using Mathematica, and the resulting graphs for the
pmf of $I$ in the case of equal probabilities for all $k$ outcomes of each trial strongly suggest
that as $n$ increases, while $k$ is left fixed, the pmf of $I$ tends toward a normal distribution.
Specifically, it appears that
\begin{equation} \label{conj}
  \frac{\sqrt{n}}{\binom n2} \frac{(I - \mu)}{\sigma} \overset{d}{\longrightarrow} N(0,1),
\end{equation} where
\begin{align}
\mu &= \binom{n}{2} \binom{k}{2} \frac{1}{k^2} = \frac{n(n-1)(k-1)}{4k},\label{meaneq}\\
\sigma^2 & = \frac{(k-1)(k+1)(n-1)n(2n+5)}{72k^2} .  \label{vareq}
\end{align}

Of course, both $\mu$ and $\sigma$ tend to $\infty$ as $n$ increases without bound, but the
preceding can be used to visualize an approximation to the distribution of $I$ for a large, fixed $n$.

The same referee suggested that it may be possible to find a result along the lines of
Theorem 5.2 in~\cite{C21}, when $n$ and $k$ are proportional.

\subsection{Moments of $I$}
 \begin{theorem} \label{MomsI}  Let $y_1, y_2,\dots, y_k$ denote fixed nonnegative integers,
 and let $n = y_1 + \cdots + y_k$.
 The first two raw moments of $I$ are
 \begin{equation}
 E(I) = \binom n2  \sum_{1\leq i< j \leq k} p_i p_j .
\end{equation}  and

\begin{multline}
E(I^2)=\binom n2 \sum_{1\leq i<j \leq k} p_i p_j 
   + 2\binom n3 \sum_{1 \leq i<j \leq k} \left( p_i^2 p_j + p_i p_j^2 \right) 
   + 6\binom n4 \sum_{1 \leq i<j \leq k} p_i^2 p_j^2       \\
   + 10\binom n3 \sum_{1 \leq h<i<j \leq k} p_h p_i p_j 
   + 36\binom n4 \sum_{1\leq g<h<i<j \leq k} p_g p_h p_i p_j \\
  + 12\binom n4 \sum_{1\leq h<i<j \leq k} 
    \left( p_h^2 p_i p_j + p_h p_i^2 p_j + p_h p_i p_j^2 \right). 
\end{multline}

Thus the variance of $I$ is given by
\begin{align}
V(I) &= \binom n2 \sum_{1\leq i<j \leq k} p_i p_j 
   + 2\binom n3 \sum_{1 \leq i<j \leq k} \left( p_i^2 p_j + p_i p_j^2 \right)  \\
   &\qquad -  \binom n2 (2n-3)  \sum_{1 \leq i<j \leq k} p_i^2 p_j^2       
   + 10\binom n3 \sum_{1 \leq h<i<j \leq k} p_h p_i p_j \notag\\
   & \qquad - 6\binom n2 (2n-3) \sum_{1\leq g<h<i<j \leq k} p_g p_h p_i p_j \notag\\
  & \qquad - 2\binom n2 (2n-3) \sum_{1\leq h<i<j \leq k} 
    \left( p_h^2 p_i p_j + p_h p_i^2 p_j + p_h p_i p_j^2 \right) .   \notag
 \end{align}   
 \end{theorem}
 
 \begin{proof}
 \begin{align*}
   E(I) & = \sum_{i\geq 0} i \cdot P(I=i) \\
           & = \sum_{i\geq 0}  \sum_{\yv\in\Comp_k(n)} i \ \inv(y_1, \dots, y_k; i)
              p_1^{y_1} \cdots p_k^{y_k} \\
           & =\sum_{\yv\in\Comp_k(n)} p_1^{y_1} \cdots p_k^{y_k}  \sum_{i \geq 0} 
              i \ \inv(y_1, \dots, y_k; i) \\
           & =\sum _{\yv\in\Comp_k(n)} \frac 12 p_1^{y_1}\cdots p_k^{y_k}  
             \binom{n}{y_1, \dots, y_k} \sum_{1 \leq i < j \leq k} y_i y_j \\
           & = \binom n2  \sum_{1\leq i< j \leq k} p_i p_j .
 \end{align*}
 \end{proof}

 \begin{multline*}
 E(I^2)  = \sum_{i\geq 0} i^2 \cdot P(I=i) \\
            = \sum_{i\geq 0}  \sum_{\yv\in\Comp_k(n)} i^2 \ \inv(y_1, \dots, y_k; i)
              p_1^{y_1} \cdots p_k^{y_k} \\
            =\sum_{\yv\in\Comp_k(n)} p_1^{y_1} \cdots p_k^{y_k}  \sum_{i \geq 0} 
              i^2\ \inv(y_1, \dots, y_k; i) \\
           = \sum_{\yv\in\Comp_k(n)} p_1^{y_1} \cdots p_k^{y_k} 
           \binom{y_1+\cdots+y_k}{y_1,\dots,y_k}  \left( 
  \frac{1}{12}  \sum_{1\leq i < j \leq k} y_i y_j 
  + \frac{1}{12} \sum_{1\leq i < j \leq k} y_i^2 y_j
  + \frac{1}{12} \sum_{1\leq i < j \leq k} y_i y_j^2 \right. \\
  + \frac{1}{4} \sum_{1\leq i < j \leq k} y_i^2 y_j^2 
  + \frac{1}{6} \sum_{1\leq h<i<j\leq k} y_h y_i y_j  
  + \frac{1}{2} \sum_{1\leq h<i<j\leq k} y_h^2 y_i y_j \\
  +\
   \left. \frac{1}{2} \sum_{1\leq h<i<j\leq k} y_h y_i^2 y_j
    + \frac{1}{2} \sum_{1\leq h<i<j \leq k} y_h y_i y_j^2 
    + \frac{3}{2} \sum_{1\leq g<h<i<j \leq k} y_g y_h y_i y_j \right)\\
 =\binom n2 \sum_{1\leq i<j \leq k} p_i p_j 
   + 2\binom n3 \sum_{1 \leq i<j \leq k} \left( p_i^2 p_j + p_i p_j^2 \right) 
   + 6\binom n4 \sum_{1 \leq i<j \leq k} p_i^2 p_j^2       \\
   + 10\binom n3 \sum_{1 \leq h<i<j \leq k} p_h p_i p_j 
   + 36\binom n4 \sum_{1\leq g<h<i<j \leq k} p_g p_h p_i p_j \\
  + 12\binom n4 \sum_{1\leq h<i<j \leq k} 
    \left( p_h^2 p_i p_j + p_h p_i^2 p_j + p_h p_i p_j^2 \right). 
 \end{multline*}

Noting that 
\begin{multline}
\left( \sum_{1\leq i<j \leq k} p_i p_j \right)^2 \\ =
 \sum_{1\leq i < j \leq k} p_i^2 p_j^2 
+2\sum_{1\leq h<i<j \leq k} \left( p_h^2 p_i p_j + p_h p_i^2 p_j + p_h p_i p_j^2 \right)
+6\sum_{1\leq g<h<i<j \leq k} p_g p_h p_i p_j,
\end{multline}

\begin{multline}
(E(I))^2 = \binom{n}{2}^2 \left( 
\sum_{1\leq i < j \leq k} p_i^2 p_j^2 
+2\sum_{1\leq h<i<j \leq k} \left( p_h^2 p_i p_j + p_h p_i^2 p_j + p_h p_i p_j^2 \right)
\right. \\
\left. +6\sum_{1\leq g<h<i<j \leq k} p_g p_h p_i p_j  \right).
\end{multline}
Thus,
\begin{align*}
V(I) &= E(I^2) - \big( E(I) \big)^2\\
       &= \binom n2 \sum_{1\leq i<j \leq k} p_i p_j 
   + 2\binom n3 \sum_{1 \leq i<j \leq k} \left( p_i^2 p_j + p_i p_j^2 \right)  \\
   &\qquad + \left( 6\binom n4 - \binom n2^2 \right) \sum_{1 \leq i<j \leq k} p_i^2 p_j^2       
   + 10\binom n3 \sum_{1 \leq h<i<j \leq k} p_h p_i p_j \\
   & \qquad + \left( 36\binom n4 - 6\binom n2^2 \right) \sum_{1\leq g<h<i<j \leq k} p_g p_h p_i p_j \\
  & \qquad + \left(12\binom n4 - 2\binom n2^2 \right) \sum_{1\leq h<i<j \leq k} 
    \left( p_h^2 p_i p_j + p_h p_i^2 p_j + p_h p_i p_j^2 \right)\\
  &= \binom n2 \sum_{1\leq i<j \leq k} p_i p_j 
   + 2\binom n3 \sum_{1 \leq i<j \leq k} \left( p_i^2 p_j + p_i p_j^2 \right)  \\
   &\qquad -  \binom n2 (2n-3)  \sum_{1 \leq i<j \leq k} p_i^2 p_j^2       
   + 10\binom n3 \sum_{1 \leq h<i<j \leq k} p_h p_i p_j \\
   & \qquad - 6\binom n2 (2n-3) \sum_{1\leq g<h<i<j \leq k} p_g p_h p_i p_j \\
  & \qquad - 2\binom n2 (2n-3) \sum_{1\leq h<i<j \leq k} 
    \left( p_h^2 p_i p_j + p_h p_i^2 p_j + p_h p_i p_j^2 \right) .   
\end{align*}

\section{Conditional distribution of $I$ given $\Yv$.}
Given a vector $\yv = (y_1, y_2, \dots, y_k)$, 
with $n = y_1 + y_2 + \cdots + y_k$,
the conditional pmf of 
$I$ given $\Yv = \yv$ is
\begin{equation}
P(I = i \ |\ \Yv = \yv) = \frac{ \inv(y_1, y_2, \dots, y_k; i )}
{\binom{n}{y_1, y_2, \dots, y_k}}
\end{equation} for $i = 0,1,2,\dots, \sum_{1\leq i<j\leq k} y_i y_j $; and 0
otherwise.

The conditional expectation is
\begin{align}
E(I=i\ |\ \Yv = \yv ) & = \binom{n}{y_1, \dots, y_k}^{-1} 
\sum_{i\geq 0} i\cdot \inv(y_1, \dots, y_k; i) \\
& = \frac{1}{2} \sum_{1\leq i<j \leq k} y_i  y_j, \notag
\end{align}
by Eq.~\eqref{mom1}.

\section{Application}
Recording the counts of the $k$ outcomes of a multinomial experiment provides valuable 
information, but does not provide \emph{all} information we might need to know.
For example, suppose that a standard six-sided die is rolled $60$ times, and it turns out
that exactly $10$ of the sixty rolls showed $j$ dots, for $j = 1, 2, 3, 4, 5, 6$.  Using that
information alone we would strongly suspect that the die was ``fair.''  However, if we
looked further and noted that the sequence of rolls was in fact
\begin{equation} \label{forw}
 \{1,1,\dots,1, 2,2,\dots,2, 3,3,\dots, 3, 4,4,\dots,4, 5,5,\dots, 5, 6,6,\dots, 6 \} 
 \end{equation}
then we would certainly think that something unexpected was happening with the die
after all.
The preceding sequence of die rolls has an inversion number of $I=0$.
Of course, the opposite extreme
\begin{equation} \label{backw}
 \{6,6,\dots,6, 5,5,\dots,5, 4,4,\dots, 4, 3,3,\dots,3, 2,2,\dots, 2, 1,1,\dots, 1 \} 
\end{equation}
would be just as suspicious.  This latter sequence has an inversion number of $I=1500$,
the maximum possible.

In contrast, a sequence of die rolls (still assuming $10$ of each of the six possible outcomes),
of around $750$, the expected value of $I$ given $\Yv = (10,10,10,10,10,10)$, might provide 
some assurance that the occurrences of the various outcomes were ``well-mixed'' as we
might expect in a truly fair die. 

For a given value $\yv$ of the multinomial random variable $\Yv$, the possible values of $I$
are from $0$ through $\sum_{1\leq i< j \leq k} y_i y_j$; it would be handy to adjust the 
value of $I$ into a scale ranging from $-1$ to $1$, with $1$ indicating that all occurrences
of outcome $1$ occur first, followed by all occurrences of outcome $2$, etc.; and with $-1$
indicating the exact reverse.  The simplest way to do this is via the transformation

 \[ H := 1 - \frac{2}{\sum_{1\leq i<j \leq k} y_i y_j} I. \]

 In this way, the sequence~\eqref{forw} corresponds to $H=1$ and the sequence~\eqref{backw}
 to $H = -1$.  
  Clearly, $E(H) = 0$.  

\section*{Acknowledgments}
The author thanks Doron Zeilberger for suggesting that he consider the
major index of a permutation during a discussion related to~\cite{S21}.  This
in turn led the author to further consider the inversion statistic of a permutation,
which in turn led to the refinement of the multinomial distribution presented in the
present manuscript.   The author further thanks the anonymous referees who
patiently provided
helpful suggestions and guidance which improved the paper, including catching
a calculation error in an earlier version.


\bibliographystyle{natbib-harv}

\appendix 
\section{Proof of Lemma~\ref{lem}}
\subsection{Some notation and elementary results}
In the theory of $q$-series, the $q$-analog $[m]_q$ of a nonnegative integer
$m$ is given by \begin{equation} \label {qint} [m]_q = 1 + q + q^2 + \cdots + q^{m-1}; 
\end{equation}
see~\citet[p. 7ff]{GR04}.
Notice $[m]_q \to m$ when $q$ is set equal to $1$.
Note further each of the following consequences:
\begin{align}
\frac{d}{dq} [m]_q & = 1 + 2q + 3q^2 + \cdots + (m-1)q^{m-2} ; \\
\frac{d}{dq}\log [m]_q & = \frac{1 + 2q + 3q^2 + \cdots + (m-1)q^{m-2}}{1+q+q^2+
\cdots+q^{m-1}} \\
  & \notag \qquad \overset{q=1}{\longrightarrow} \frac{\binom{m}{2}}{m}  = \frac 12 (m-1);
\\
\frac{d}{dq} \left( q \frac{d}{dq}\log [m]_q \right) &=
\frac{d}{dq} \frac{q + 2q^2+ 3q^3 + \cdots + (m-1) q^{m-1}}{1+q+q^2+\cdots+q^{m-1}} \\
& = \frac{ [m]_q \sum_{i=0}^{m-2} (i+1)^2 q^i -
 \left( \sum_{j=1}^{m-1} jq^j \right) \left(\sum_{r=0}^{m-2}
  (r+1)q^r \right)} { \left( [m]_q \right)^2} \notag\\
&\notag \qquad \overset{q=1}{\longrightarrow}
  \left(  m\frac{ (2m-1)m(m-1)}{6} - \frac{ m^2(m-1)^2}{4} \right) \frac{1}{m^2} \\
  &\qquad\quad= \frac{1}{12}(m^2 - 1). \notag
 \end{align}

\subsection{Proof of Lemma~\ref{lem}}
Let $y_1, y_2, \dots, y_k$ denote fixed nonnegative integers, let
$n = y_1 + y_2 + \cdots + y_k$, and let
\[ f = f(q) = \gp{y_1 + \cdots + y_k}{y_1, \dots, y_k}{q} = 
\prod_{i=1}^n [i]_q \div \prod_{r=1}^k \prod_{i_r=1}^{y_r} [i_r]_q . \]
This latter expression for the $q$-multinomial co\"efficient is equivalent to
that which was given in~\eqref{qmc}, but is more convenient for our present
purposes.

It follows that
\begin{equation} \label{log}
  \log f = \sum_{i=1}^n \log [i]_q - \sum_{r=1}^k \sum_{i_r=1}^{y_r} \log [i_r]_q
\end{equation} and thus
\begin{equation}
  \frac{d}{dq} \log f = \sum_{i=1}^n \frac{ \frac{d}{dq} [i]_q}{[i]_q} 
- \sum_{r=1}^k \sum_{i_r=1}^{y_r} \frac{ \frac{d}{dq} [i_r]_q }{[i_r]_q }.
\end{equation}

Using logarithmic differentiation, we find that
\begin{equation} \label{logdiff}
 \frac{df}{dq} = f \frac{d}{dq} \log f,
\end{equation}
and then applying Leibniz' rule, 
\begin{equation} \label{dqd}
  \frac{d}{dq}\left( q \frac{df}{dq} \right) =
  f \left( \frac{d}{dq} \left( q \frac{d}{dq}\log f \right) + q \left( \frac{d}{dq} \log f\right)^2
  \right).
\end{equation}

We proceed to evaluate the expressions contained in the right members of
\eqref{logdiff} and \eqref{dqd} at $q=1$.

\begin{align*}
\frac{d}{dq} \log f &= \sum_{i=1}^n \frac{d}{dq}\log [i]_q - 
\sum_{r=1}^k \sum_{i_r=1}^{y_r} \frac{d}{dq} \log [i_r]_q \\
 & \overset{q=1}{\longrightarrow}  
 \frac12\sum_{i=1}^n (i-1) - \frac12 \sum_{r=1}^k \sum_{i_r=1}^{y_r}
 (i_r-1) \\
 & = \frac 12 \left( \binom{n+1}{2} - n \right) 
 - \frac 12 \left(  \sum_{r=1}^k \binom{y_r + 1}{2} - y_r \right) \\
  & = \frac 14 \left( n(n+1) - y_1(y_1+1) - y_2(y_2+1) - \cdots - y_k(y_k+1) \right)\\
  & = \frac 12 \sum_{1\leq i < j \leq k} y_i y_j.
\end{align*}
Thus, by~\eqref{logdiff},
\begin{equation} \label{moment1} 
\sum_{i\geq 0} i\cdot\inv(y_1, \dots, y_k ;i) 
= q\frac{df}{dq} = \frac 12 \binom{n}{y_1, \dots, y_k} \sum_{1\leq i<j \leq k} y_i y_j, 
\end{equation}  which is~\eqref{mom1} of Lemma~\ref{lem}.

Next, we need to compute $\frac{d}{dq}\left( q \frac{d}{dq} \log f \right)$ at $q=1$.

\begin{align*}
&\phantom{====} \frac{d}{dq} \left( q \frac{d}{dq} \log f \right) \\ &= 
   \sum_{i=1}^n \frac{d}{dq} \left( q \frac{d}{dq} \log [i]_q  \right)
 - \sum_{r=1}^k \sum_{i_r=1}^{y_r} \frac{d}{dq} \left( q \frac{d}{dq} \log [i_r]_q  \right) \\
 &\overset{q=1}{\longrightarrow} \frac{1}{12} \left\{  
 \left(\frac{n(n+1)(2n+1)}{6} - n\right) - \sum_{r=1}^k\left( \frac{y_r(y_r+1)(2y_r+1)}{6}
 - y_r \right) \right\} \\
&=  \frac{1}{12} \left\{ \sum_{1\leq i < j \leq k} y_i y_j + \sum_{1\leq i < j \leq k} y_i^2 y_j
  + \sum_{1\leq i < j \leq k} y_i y_j^2 + 2 \sum_{1\leq h<i<j\leq k} y_h y_i y_j \right\}.
\end{align*}

So, we have
\begin{multline} \label{moment2}
\sum_{i\geq 0} i^2 \cdot \inv(y_1, \dots, y_k; i) = \frac{d}{dq} \left(q \frac{df}{dq} \right) 
 \\ = f \left( \frac{d}{dq} \left( q \frac{d}{dq} \log f \right) + q \left( \frac{d}{dq} \log f \right)^2
 \right) \\
 = \binom{y_1+\cdots+y_k}{y_1,\cdots,y_k}  \left( 
  \frac{1}{12} \left\{ \sum_{1\leq i < j \leq k} y_i y_j + \sum_{1\leq i < j \leq k} y_i^2 y_j
  + \sum_{1\leq i < j \leq k} y_i y_j^2 + \right. \right. \\
  \left. \left. 2\cdot\sum_{1\leq h<i<j\leq k} y_h y_i y_j \right\}  + \frac 14
 \left( \sum_{1\leq i<j\leq k} y_i y_j  \right)^2 \right) \\
= \binom{y_1+\cdots+y_k}{y_1,\cdots,y_k}  \left( 
  \frac{1}{12}  \sum_{1\leq i < j \leq k} y_i y_j 
  + \frac{1}{12} \sum_{1\leq i < j \leq k} y_i^2 y_j
  + \frac{1}{12} \sum_{1\leq i < j \leq k} y_i y_j^2 \right.  \\
  + \frac{1}{4} \sum_{1\leq i < j \leq k} y_i^2 y_j^2 
  + \frac{1}{6} \sum_{1\leq h<i<j\leq k} y_h y_i y_j  
  + \frac{1}{2} \sum_{1\leq h<i<j\leq k} y_h^2 y_i y_j \\+
      \left. \frac{1}{2} \sum_{1\leq h<i<j\leq k} y_h y_i^2 y_j
    + \frac{1}{2} \sum_{1\leq h<i<j \leq k} y_h y_i y_j^2 
    + \frac{3}{2} \sum_{1\leq g<h<i<j \leq k} y_g y_h y_i y_j \right),
\end{multline} which is~\eqref{mom2} of Lemma~\ref{lem}. \qed

\end{document}